\documentclass[10pt]{article}
\usepackage{latexsym,amsfonts,amsthm,amsmath,amssymb,cases}
\usepackage{cite}
\usepackage{enumerate}
\usepackage{graphicx}
\usepackage{color}
\DeclareMathOperator{\esssup}{ess\,sup}
\newtheorem{thm}{Theorem}[section]

\newtheorem{lemma}{Lemma}[section]
\newtheorem{prob}{Problem}[section]
\newtheorem{subprob}{Subroblem}[section]
\newtheorem{assumption}{Assumption}[section]
\newtheorem{remark}{Remark}[section]

\newtheorem{defn}{Definition}[section]
\newtheorem{example}{Example}[section]
\newcounter{nextauthor}
\def\mathrm{\mbox}

\numberwithin{remark}{section}

\begin{document}
\title{{\Large \bf A Linear-quadratic Mean-Field Stochastic Stackelberg Differential Game with Random Exit Time}\thanks{This work was supported by the National Natural Science Foundation of China (11471230, 11671282).}}
\author{Zhun Gou$^a$, Nan-jing Huang$^a$\footnote{Corresponding author.  E-mail addresses: nanjinghuang@hotmail.com; njhuang@scu.edu.cn} and Ming-hui Wang$^b$\\
{\small\it a. Department of Mathematics, Sichuan University, Chengdu, Sichuan 610064, P.R. China}\\
{\small\it b. Department of Economic Mathematics, Southwestern University of Finance and Economics,}\\
{\small\it Chengdu, Sichuan 610074, P.R. China}}
\date{}
\maketitle
\begin{center}
\begin{minipage}{5.5in}
\noindent{\bf Abstract.} In this paper, we investigate a new model of a linear-quadratic mean-field stochastic Stackelberg differential game with one leader and two followers, in which the leader is allowed to stop her strategy at a random time. Our overarching goal is to find the Stackelberg solution of the leader and followers for such a model. By employing the backward induction method, the state equation is divided into two-stage equations. Moreover, by using the maximum principle and the verification theorem, the Stackelberg solution is obtained for such a model.
\end{minipage}
\end{center}

\section{Introduction}
\paragraph{}
The study of Stackelberg game was pioneered by Stackelberg \cite{Stackelberg1952The}, where there are two players in the game. One player acts as the leader (she) while the other behaves as the follower (he). First, the leader announces her strategy and the follower reacts to it by optimizing his objective function in accordance with the leader's announced strategy. Then, the leader would like to seek a strategy to optimize her cost function based on the follower's best response. The best strategy of the leader together with the best response of the follower is known as a Stackelberg solution. Since then, Stackelberg game has been studied extensively by many authors because of its wide application in various fields, including economics and finance, management and decision, transportation and evolutionary biology (see, for example,  \cite{Carreno2019Stackelberg, Megahed2019the, Moon2021linear, Mu2018Stackelberg, zou2020A}). We would like to point out that many existed literatures only investigated the Stackelberg game with one leader and one follower. However, it is important to consider the Stackelberg game with multiple followers, which is more complicate but obviously more suitable for describing some practical problems. Different from the Stackelberg game with one leader and one follower, the strategies of all the players in the Stackelberg game with multiple followers should be impacted by all the other players' strategies (see, for instance, \cite{Askar2018Tripoly, Fang2017Coordinated, Jiang2018Multi}).

As an extension of classical Stackelberg game, the mean-field stochastic Stackelberg differential game (MF-SSDG), which is described by mean-field stochastic differential equations (MF-SDEs), has attracted much attention recently \cite{Bensoussan2017Linear, Bensoussan2020Mean, huang2020Linear, Lin2019feedback, Lin2018An, Moon2018linear, Si2021Backward, Wang2020An}. A significant feature of the game is that  not only the state variable and the controls but also their expectations are involved in the state equation and objective functions. Such a feature originates from the mean-field theory, which was developed to study the collective behaviors resulting from individuals mutual interactions in various physical and sociological dynamical systems. The current paper focuses on the study of MF-SSDG with multi-follower.

On the other hand, it is well known that default risk, which is closely related to default events, naturally appears in financial markets \cite{Aksamit2017Enlargement, Calvia2020Risk, Jeanblanc2020Characteristics}. The standard approach to model default risk is to use the theory of enlargement of filtration, which has been extensively studied by many authors (see, for example, \cite{Aksamit2017Enlargement, Jeanblanc2009CMathematical, Peng2009BSDE, Pham2010Stochastic}). Besides, for practical reasons, one may require considering default risk in economic models, such as the retailer-supplier uncooperative replenishment model \cite{Wu2017A} and supply chain with one retailer and several suppliers \cite{Babich2007SCompetition}. In \cite{Babich2007SCompetition, Wu2017A}, both of models were described by employing the stochastic Stackelberg differential game (SSDG). Moreover, in \cite{Peng2009BSDE},  the problem of zero-sum stochastic differential game with default risk was investigated, which can be seen as a utility maximization problem under the model uncertainty. Thus, it is necessary to consider the default risk in stochastic differential games. Especially in some real situations, the leader may stop her action at a "surprising" time (usually a random time), which cannot be read or predicted from the reference observation. For such a "surprising" time, we call it a random exit time which can be regarded as a default time. Let us illustrate this case with the following example.
\begin{example}\label{87}
Following \cite{Novak2021a}, we assume that the  dynamics of the resource stock $x(t)$ can be written as
\begin{align*}
dx(t)=[rx(t)-h(u(t),v_1(t),v_2(t))]dt,\quad x(0)=x_0>0,
\end{align*}
which does not only depend on the intensity of attacks from two terror organizations (the followers), but it is also influenced
by the counterterror measures from the government (the leader). Here $x_0$ denotes the initial stock of resources for terrorists; $u(t)$ is the counterterror measures of the government; $v_1(t)$ and $v_2(t)$ are  the intensity of attacks from terrorists. Moreover, we assume that along a trajectory the following non-negativity constraint applies:
$$
x(t)\geq0,\quad \mbox{for all $t\geq0$}.
$$
Some suitable conditions are imposed on $h(u(t),v_1(t),v_2(t))$. The objective function of the government is
$$
J(u(t),v_1(t),v_2(t))=\int_0^{T}e^{-\rho t}[\varpi h(u(t),v_1(t),v_2(t))-cx(t)-k_1v_1(t)-k_2v_2(t)-au(t)]dt+e^{-\rho T}kx(T),
$$
and the objective function of the terror organization is
$$
J_i(v_1(t),v_2(t))=\int_0^{T}e^{-\rho_i t}[\sigma_ix(t)+\beta_iv_i(t)]dt+e^{-\rho_i T}c_ix(T),
$$
where $\rho,\rho_i,c,c_i,k_i,\sigma_i,\beta_i$ ($i=1,2$) are all  positive constants. First, the government announce her counterterror measures $u(t)$. Then, a Nash game is considered for the two terror organizations, i.e., they would like to maximize their objective functions by
\begin{equation*}
\begin{cases}
J_1(v^*_1(t),v^*_2(t))=\sup\limits_{v_1(t)\geq 0}J_1(v_1(t),v^*_2(t)),\\
J_2(v^*_1(t),v^*_2(t))=\sup\limits_{v_2(t)\geq 0}J_2(v^*_1(t),v_2(t)).
\end{cases}
\end{equation*}
Finally, considering that the terror organizations would take strategy $(v^*_1,v^*_2)$, the government would like to maximize her objective function such that
$$
J(u^*(t),v^*_1(t),v^*_2(t))=\sup_{u(t)\geq 0}J(u(t),v^*_1(t),v^*_2(t)).
$$

Obviously, this is a problem of the Stackelberg game with one leader and two followers. Nevertheless, when an emergency happens and is made a top priority, the government has to stop her counterterror measures at a random time $\tau$ and deal with the emergency. In this case the objective function of the government becomes
$$
J(u(t),v_1(t),v_2(t))=\int_0^{\tau}e^{-\rho t}[\varpi h(u(t),v_1(t),v_2(t))-cx(t)-k_1v_1(t)-k_2v_2(t)-au(t)]dt+e^{-\rho \tau}ky(\tau),
$$
where $y(\tau)$ represents the discounted value of $x(\tau)$, i.e., the government suffers a loss at time $\tau\in[0,T]$. As a result,  the Stackelberg game with one leader and two followers mentioned above is nothing but a form MF-SSDG with one leader and two followers in which the leader is allowed to exit at a random time (for details, see Problem \ref{32} in Section 2).
\end{example}

Example \ref{87} tells us that, in some practical situations, it is necessary and important to consider the problem of MF-SSDG with random exit time.   The purpose of this paper is to investigate a new model of MF-SSDG with random exit time, which has never been yet explored in the previous literature. The main features of this paper can be summarized as follows: (i) Based on the technique of progressive enlargement of filtration, we construct a new model of a linear-quadratic mean-field stochastic Stackelberg differential game (LQ-MF-SSDG) with one leader and two followers, in which the leader is allowed to stop her strategy at a random time;  (ii) The state process in the model is governed by an MF-SDE divided into two-stage subequations by random exit time, in which the coefficients of subequations are allowed to be different for describing practical problems; (iii) By employing the backward induction method, the  Stackelberg solution is obtained for LQ-MF-SSDG with random exit time, even though random coefficient is involved in the mean-field term.

The rest of this paper is structured as follows.  In the next section, we give the formulation of LQ-MF-SSDG with random exit time. In Section 3, we derive the  Stackelberg solution for the leader and followers of LQ-MF-SSDG with random exit time. Finally, we make some concluding remarks in Section 4.

\section{Problem Formulation}
Consider the linear MF-SDE system in the complete probability space $(\Omega,\mathfrak{F},\mathfrak{F}_t,\mathbb{P})$ satisfying the usual hypothesis as follows:
\begin{equation}\label{SDE1}
\begin{cases}
dX(t)=\left[a(t)X(t)+\overline{a}(t)\overline{X}(t)+b_1(t)v_1(t)+b_2(t)v_2(t)+b(t)v_0(t)\right]dt\\
\qquad\qquad +\left[c(t)X(t)+\overline{c}(t)\overline{X}(t)\right]dB_t, \quad t\in[0,T],\\
X(0)=x_0,
\end{cases}
\end{equation}
where $T>0$ is a finite time duration; $\tau$ is a non-negative random variable, which represents the possible exit time of the leader; $X(t)$ is the state process; $v_0(t)$ is the the control process of the leader; $v_1(t)$ and $v_2(t)$ are the control processes of the followers $\mathcal{P}_1$ and $\mathcal{P}_2$, respectively; $\overline{X}(t)$ is the mean-field term with
$$
\overline{X}(t)=\mathbb{E}\left[X(t)\Big|\mathfrak{F}_{0}\right]\mathbb{I}_{t\in[0,\tau)}+\mathbb{E}\left[X(t)\Big|\mathfrak{F}_{\tau}\right]\mathbb{I}_{t\in[\tau,T]};
$$
the $\sigma$-algebra $\mathcal{F}=(\mathcal{F}_t)_{t\geq0}$, generated by the standard one-dimensional Brownian motion $B_t$, is right-continuous and increasing; $x_0$ is an $\mathcal{F}_0$-measurable and square integrable random variable; $\mathfrak{F}$ is the smallest right continuous extension in which $\tau_0$ becomes an $\mathfrak{F}$-stopping time, i.e.,
$\mathfrak{F}_{s}=\mathcal{F}_s \vee \sigma\left(\mathbb{I}_{\tau_0\leq u},u\in [0,s]\right)$, for all $s>0$.

Clearly, $\tau=\tau_0\wedge T\in[0,T]$ is also an $\mathfrak{F}$-stopping time. Noticing that when $\tau_0=\tau$ (i.e., the leader stops her action before the game ends), the controller $v_0$ contributes nothing on the time interval $(\tau,T]$. Thus,  it is natural to require that $b(t)=b_0(t)\mathbb{I}_{t\in[0,\tau]}$.

For convenience, we set $i=1,2$, $j=0,1,2$ and $\mathcal{L}_{\mathfrak{F}}^2(0,T;\mathbb{R})$ the space of all $\mathfrak{F}$-adapted and square integrable processes throughout this paper. Now we introduce the following objective functions.
\begin{defn}
The objective functions of the follower $\mathcal{P}_i$ and the leader are given by
\begin{equation}\label{33}
J_i(v_1,v_2,v_0)=-\frac{1}{2}\mathbb{E}\left[\int_0^{T}p_i(t)X^2(t)+\overline{p}_i(t)\overline{X}^2(t)+q_i(t)v_i^2(t)dt+r_i(T)X^2(T)\Big|\mathfrak{F}_{0}\right],
\end{equation}
and
\begin{equation}\label{34}
J_0(v_1,v_2,v_0)=-\frac{1}{2}\mathbb{E}\left[\int_0^{\tau}p_0(t)Y^2(t)+\overline{p}_0(t)\overline{Y}^2(t)+q_0(t)v_0^2(t)dt+r_0(\tau)Y^2(\tau)\Big|\mathfrak{F}_{0}\right],
\end{equation}
respectively. Here $Y(t)$ is the solution to the following MF-SDE with single jump at $t=\tau$.
\begin{equation}\label{31}
\begin{cases}
dY(t)=\left[a(t)Y(t)+\overline{a}(t)\overline{Y}(t)+b_1(t)v_1(t)+b_2(t)v_2(t)+b_0(t)v_0(t)\right]dt\\
\qquad\qquad + \left[c(t)Y(t)+\overline{c}(t)\overline{Y}(t)\right]dB_t+\left[d_0(t)Y(t)+\overline{d}_0(t)\overline{Y}(t)\right]d\mathbb{I}_{\tau\leq t}, \quad t\in(0,\tau],\\
Y(0)=x_0.
\end{cases}
\end{equation}
\end{defn}
\begin{remark}
It is easy to show that $Y(t)=X(t)$ a.s. for $t\in[0,\tau)$, while there is a penalty at time $\tau$ for the terminal gain of the leader.
\end{remark}
Next we introduce the admissible control sets and the optimal strategy problem.
\begin{prob}\label{32}
Consider the admissible control sets:
\begin{align*}
\mathcal{V}[0,T]=&\left\{v:[0,T]\times \Omega\rightarrow \mathbb{R}\Big|v(t)\;\mbox{is $\mathfrak{F}$-progressively measurable with}\; \mathbb{E}\left[\int_0^T|v(t)|^2dt\Big|\mathfrak{F}_{0}\right]<\infty\;\mbox{a.s.} \right\};\\
\mathcal{V}[0,\tau]=&\left\{v:[0,\tau]\times \Omega\rightarrow \mathbb{R}\Big|v(t)\;\mbox{is $\mathfrak{F}$-progressively measurable with}\; \mathbb{E}\left[\int_0^{\tau}|v(t)|^2dt\Big|\mathfrak{F}_{0}\right]<\infty\;\mbox{a.s.} \right\};\\
\mathcal{V}[\tau,T]=&\left\{v:[\tau,T]\times \Omega\rightarrow \mathbb{R}\Big|v(t)\;\mbox{is $\mathfrak{F}$-progressively measurable with}\; \mathbb{E}\left[\int_{\tau}^T|v(t)|^2dt\Big|\mathfrak{F}_{\tau}\right]<\infty\;\mbox{a.s.} \right\}.
\end{align*}
The followers would like to find the Nash equilibrium point $(v^*_1,v^*_2)\in \mathcal{V}[0,T]\times \mathcal{V}[0,T]$ such that
\begin{equation}\label{35}
\begin{cases}
J_1(v^*_1,v^*_2,v_0)=\mathop{\esssup}\limits_{v_{1} \in \mathcal{V}[0,T]}J_1(v_1,v^*_2,v_0),\\
J_2(v^*_1,v^*_2,v_0)=\mathop{\esssup}\limits_{v_{2} \in \mathcal{V}[0,T]}J_2(v^*_1,v_2,v_0).
\end{cases}
\end{equation}
 Considering that the followers would take strategy $(v^*_1,v^*_2)$, the leader expects to find $v^*_0\in \mathcal{V}[0,\tau]$ such that
\begin{equation}\label{38}
J_0(v^*_1,v^*_2,v^*_0)=\mathop{\esssup}\limits_{v_{0}\in \mathcal{V}[0,\tau]}J_0(v^*_1,v^*_2,v_0).
\end{equation}
\end{prob}
In order to obtain our main results, we make the following assumptions for MF-SDEs \eqref{SDE1}, \eqref{31}, and objective functions \eqref{33}, \eqref{34} throughout this paper.
\begin{assumption}
For $\theta\in\{a,\overline{a},c,\overline{c}\}$, $\vartheta_i\in\{b_i,p_i,\overline{p}_i,q_i,r_i\}$, suppose that $\theta$ and $\vartheta_i$ have the following decompositions.
$$
\theta(t)=\theta_1(t)\mathbb{I}_{t\in[0,\tau)}+\theta_2(t)\mathbb{I}_{t\in[\tau,T]},\quad
\vartheta_i(t)=\vartheta_{i,1}(t)\mathbb{I}_{t\in[0,\tau)}+\vartheta_{i,2}(t)\mathbb{I}_{t\in[\tau,T]}.
$$
In addition, suppose that the following conditions hold.
\begin{enumerate}[($A$)]
\item $a_i,\overline{a}_i,b_0,b_{i,1},b_{i,2},c_i,\overline{c}_i,d_0,\overline{d}_0$ are all deterministic, uniformly bounded and real-valued functions on $[0,T]$;

\item $p_0,\overline{p}_0,q_0,r_0,p_{i,j},\overline{p}_{i,j},q_{i,j},r_{i,j}$ are all deterministic, continuous and positive functions on $[0,T]$;

\item $q^{-1}_{1,i}b_{1,i}^2=q^{-1}_{2,i}b_{2,i}^2=l_i$;

\item (Hypothesis $(\mathcal{H})$) Every c\`{a}dl\`{a}g $\mathcal{F}$-martingale remains an $\mathfrak{F}$-martingale;

\item
Suppose that there exists an $\mathfrak{F}$-predictable (respectively,  $\mathcal{F}$-predictable) process $\gamma^{\mathfrak{F}}$ (respectively, $\gamma^{\mathcal{F}}$) with $\gamma^{\mathfrak{F}}(s)=\mathbb{I}_{s<\tau}\gamma^{\mathcal{F}}(s)$ such that
\begin{equation}\label{37}
A(t)=\mathbb{I}_{\tau\leq t}-\int_0^t\gamma^{\mathfrak{F}}(s)ds=\mathbb{I}_{\tau\leq t}-\int_0^{t\wedge \tau}\gamma^{\mathcal{F}}(s)ds \quad (t\in [0,T])
\end{equation}
is an $\mathfrak{F}$-martingale with jump time $\tau$. The process $\gamma^{\mathfrak{F}}$ (respectively, $\gamma^{\mathcal{F}}$) is called the $\mathfrak{F}$-intensity (respectively, $\mathcal{F}$-intensity) of $\tau$. In addition, we suppose that $\gamma^{\mathfrak{F}}$ is upper bounded.
\end{enumerate}
\end{assumption}

\begin{remark}
Concerning assumptions $(A)$-$(E)$, some remarks are listed as follows: (i) $(A)$ implies that for any $(v_0,v_1,v_2)\in\mathcal{V}[0,\tau]\times\mathcal{V}[0,T]\times\mathcal{V}[0,T]$, \eqref{SDE1} admits a unique solution $X(t)\in \mathcal{L}_{\mathfrak{F}}^2(0,T;\mathbb{R})$; (ii) $(B)$ ensures the concavity of the objective functionals; (iii) $(C)$ ensures the well-posedness of related Riccati equations for followers; (iv) $(D)$ and $(E)$ are classical assumptions in the theory of progressive enlargement.
\end{remark}

\section{Main results}
In this section, we focus on searching the Stackelberg solution of LQ-MF-SSDG with random exit time, which is closely related to the optimal control for SDEs with default (with respect to the study of SDEs with default, we refer the reader to \cite{Bachir2020Stochastic, Kharroubi2014Progressive, Peng2009BSDE}). Motivated by the backward induction method (see, for example, \cite{gou2020optimal, Pham2010Stochastic}), we first decompose the state equation into the following two-stage equations on $t\in[0,\tau]$ (the first stage) and $t\in[\tau,T]$ (the second stage):
\begin{equation*}
\begin{cases}
dX_1(t)=\left[a_1(t)X_1(t)+\overline{a}_1(t)\overline{X}_1(t)+b_{1,1}(t)v_{1,1}(t)+b_{2,1}(t)v_{2,1}(t)+b_0(t)v_{0}(t)\right]dt\\
\qquad\qquad\;+\left[c_1(t)X_1(t)+\overline{c}_1(t)\overline{X}_1(t)\right]dB_t, \quad t\in(0,\tau],\\
dX_2(t)=\left[a_2(t)X_2(t)+\overline{a}_2(t)\overline{X}_2(t)+b_{1,2}(t)v_{1,2}(t)+b_{2,2}(t)v_{2,2}(t)\right]dt\\
\qquad\qquad\;+\left[c_2(t)X_2(t)+\overline{c}_2(t)\overline{X}_2(t)\right]dB_t, \quad t\in(\tau,T],\\
X_1(0)=x_0,\quad X_2(\tau)=X_1(\tau).
\end{cases}
\end{equation*}
We note that Lemma 2.1 of \cite{Pham2010Stochastic} ensures the rationality of the above decomposition. Further, we decompose the controls of the followers as follows:
$$
v_i(t)=v_{i,1}(t)\mathbb{I}_{t\in[0,\tau)}+v_{i,2}(t)\mathbb{I}_{t\in[\tau,T]}.
$$
 Next, for fixed $v_0\in\mathcal{V}[0,\tau]$, we aim to find $(v^*_{1,2},v^*_{2,2})\in \mathcal{V}[\tau,T]\times \mathcal{V}[\tau,T]$ such that
\begin{equation}\label{22}
\begin{cases}
\quad\; J_{1,2}(v^*_{1,2},v^*_{2,2})=\mathop{\esssup}\limits_{v_{1,2} \in \mathcal{V}[\tau,T]}J_{1,2}(v_{1,2},v^*_{2,2})\\
=\mathop{\esssup}  \limits_{v_{1,2} \in \mathcal{V}[\tau,T]}\Big\{-\frac{1}{2}\mathbb{E}\left[\int_{\tau}^{T}p_{1,2}(t)X_2^2(t)+\overline{p}_{1,2}(t)\overline{X}_2^2(t)+q_{1,2}(t)v_{1,2}^2(t)dt+r_{1}(T)X_2^2(T)\Big|\mathfrak{F}_{\tau} \right]\Big\},\\
\quad\; J_{2,2}(v^*_{1,2},v^*_{2,2})=\mathop{\esssup}\limits_{v_{2,2} \in \mathcal{V}[\tau,T]}J_{2,2}(v^*_{1,2},v_{2,2})\\
=\mathop{\esssup}  \limits_{v_{2,2} \in \mathcal{V}[\tau,T]}\Big\{-\frac{1}{2}\mathbb{E}\left[\int_{\tau}^{T}p_{2,2}(t)X_2^2(t)+\overline{p}_{2,2}(t)\overline{X}_2^2(t)+q_{2,2}(t)v_{2,2}^2(t)dt+r_{2}(T)X_2^2(T)\Big|\mathfrak{F}_{\tau} \right]\Big\},
\end{cases}
\end{equation}
and $(v^*_{1,1},v^*_{2,1})\in \mathcal{V}[0,\tau]\times \mathcal{V}[0,\tau]$ such that
\begin{equation}\label{23}
\begin{cases}
\quad\; J_{1,1}(v^*_{1,1},v^*_{2,1},v_0)=\mathop{\esssup}\limits_{v_{1,1} \in \mathcal{V}[0,\tau]}J_{1,1}(v_{1,1},v^*_{2,1},v_0)\\
=\mathop{\esssup}  \limits_{v_{1,1} \in \mathcal{V}[0,\tau]}\mathbb{E}\Big\{\int_0^{\tau}-\frac{1}{2}\left[p_{1,1}(t)X_1^2(t)+\overline{p}_{1,1}(t)\overline{X}_2^2(t)+q_{1,1}(t)v_{1,1}^2(t)\right]dt+J_{1,2}(X_2^*(t),v^*_{1,2}(t))\Big|\mathfrak{F}_{0} \Big\},\\
\quad\; J_{2,1}(v^*_{1,1},v^*_{2,1},v_0)=\mathop{\esssup}\limits_{v_{2,1} \in \mathcal{V}[0,\tau]}J_{2,1}(v^*_{1,1},v_{2,1},v_0)\\
=\mathop{\esssup}  \limits_{v_{2,1} \in \mathcal{V}[0,\tau]}\mathbb{E}\Big\{\int_0^{\tau}-\frac{1}{2}\left[p_{2,1}(t)X_1^2(t)+\overline{p}_{2,1}(t)\overline{X}_2^2(t)+q_{2,1}(t)v_{2,1}^2(t)\right]dt+J_{2,2}(X_2^*(t),v^*_{2,2}(t))\Big|\mathfrak{F}_{0} \Big\}.
\end{cases}
\end{equation}
Let
$$
\widetilde{v}_i(t)=v^*_{i,1}(t)\mathbb{I}_{t\in[0,\tau)}+v^*_{i,2}(t)\mathbb{I}_{t\in[\tau,T]}.
$$
According to \eqref{35} and Theorem 4.1 in \cite{gou2020optimal}, we know that $(\widetilde{v}_1,\widetilde{v}_2)\in \mathcal{V}[0,T]\times\mathcal{V}[0,T]$ is the global Nash equilibrium point, i.e., $(\widetilde{v}_1,\widetilde{v}_2)=(v_1^*,v_2^*)$. Moreover,
$$
J_i(v^*_1,v^*_2,v_0)=J_{i,1}(v^*_{1,1},v^*_{2,1},v_0).
$$

Assume that followers take strategy $(\widetilde{v}_1,\widetilde{v}_2)$, we proceed to find the optimal strategy $v_0^*$ for the leader.
In the following subsections, we solve $(v^*_{1,2},v^*_{2,2})$, $(v^*_{1,1},v^*_{2,1})$ and $v_0^*$ in the sequence.
\subsection{Nash equilibrium for followers at the second stage}
In this subsection, we apply the Pontryagin-type maximum principle to find Nash equilibrium point $(v^*_{1,2},v^*_{2,2})$ for followers at the second stage. We restate problem \eqref{22} as follows.
\begin{subprob} \label{39}
Find $(v^*_{1,2},v^*_{2,2})\in \mathcal{V}[\tau,T]\times \mathcal{V}[\tau,T]$ such that
\begin{align*}
J_{1,2}(v^*_{1,2},v^*_{2,2})&=\mathop{\esssup}\limits_{v_{1,2} \in \mathcal{V}[\tau,T]}J_{1,2}(v_{1,2},v^*_{2,2}),\\
J_{2,2}(v^*_{1,2},v^*_{2,2})&=\mathop{\esssup}\limits_{v_{2,2} \in \mathcal{V}[\tau,T]}J_{2,2}(v^*_{1,2},v_{2,2}).
\end{align*}
\end{subprob}

To solve this subproblem, we define the Hamiltonian function of follower $\mathcal{P}_i$ at second stage by
\begin{align*}
H_{i,2}&:[\tau,T]\times \mathbb{R}\times \mathbb{R}\times \mathbb{R}\times \mathbb{R}\times \mathbb{R}\times \mathbb{R}\rightarrow \mathbb{R},\\
H_{i,2}&=H_{i,2}(t,X_2,\overline{X}_2,v_{1,2},v_{2,2},P_{v_{i,2}},Q_{v_{i,2}})\\
&=-\frac{1}{2}\left[p_{i,2}X_2^2+\overline{p}_{i,2}\overline{X}_2^2+q_{i,2}v_{i,2}^2\right]+\left[a_2X_2+\overline{a}_2\overline{X}_2+b_{1,2}v_{1,2}+b_{2,2}v_{2,2}\right]P_{v_{i,2}}\\
&\quad\mbox{}\,+\left[c_2X_2+\overline{c}_2\overline{X}_2\right]Q_{v_{i,2}}.
\end{align*}
By the sufficient and necessary maximum principle of MF-SDEs with default (see \cite{gou2020optimal}), we have
$$
0=\frac{\partial H_{i,2}}{\partial v_{i,2}}\Big|_{v_{i,2}=v^*_{i,2}}=-q_{i,2}v^*_{i,2}+b_{i,2}P_{v^*_{i,2}}\quad \Rightarrow \quad v^*_{i,2}=q^{-1}_{i,2}b_{i,2}P_{v^*_{i,2}},
$$
where $P_{v^*_{i,2}}$ is the solution to corresponding linear BSDE. This leads to the following theorem.

\begin{thm}\label{7}
$(v^*_{1,2},v^*_{2,2})$ is a Nash equilibrium point for Subproblem \ref{39} if and only if
\begin{equation}\label{80}
(v^*_{1,2},v^*_{2,2})=(q^{-1}_{1,2}b_{1,2}P_{v^*_{1,2}},q^{-1}_{2,2}b_{2,2}P_{v^*_{2,2}}),
\end{equation}
where the triple $(X_2^*,P_{v^*_{i,2}},Q_{v^*_{i,2}})$ satisfies the following mean-field forward-backward stochastic differential equation (MF-FBSDE):
\begin{equation}\label{1}
\begin{cases}
dX^*_2=\left[a_2X^*_2+\overline{a}_2\overline{X}^*_2+q^{-1}_{1,2}b_{1,2}^2P_{v^*_{1,2}}+q^{-1}_{2,2}b_{2,2}^2P_{v^*_{2,2}}\right]dt+\left[c_2X^*_2+\overline{c}_2\overline{X}^*_2\right]dB_t,\quad t\in(\tau,T],\\
dP_{v^*_{i,2}}=\left[p_{i,2}X^*_2+\overline{p}_{i,2}\overline{X}^*_2-a_2P_{v^*_{i,2}}-\overline{a}_2\overline{P}_{v^*_{i,2}}-c_2Q_{v^*_{i,2}}-\overline{c}_2\overline{Q}_{v^*_{i,2}}\right]dt+Q_{v^*_{i,2}}dB_t, \quad t\in[\tau,T),\\
X^*_2(\tau)=X_1(\tau),\quad P_{v^*_{i,2}}(T)=-r_{i}(T)X^*_2(T).
\end{cases}
\end{equation}
\end{thm}

According to \eqref{23}, before solving the Nash equilibrium point for followers at the first stage, we need to compute the optimal objective subfunctions $J_{1,2}(X_2^*(t),v^*_{1,2}(t))$ and $J_{2,2}(X_2^*(t),v^*_{2,2}(t))$. First, we focus on the feedback representation of the solutions to \eqref{1}.

Observing the terminal condition of $P_{v^*_{i,2}}$, we set
\begin{equation}\label{3}
P_{v^*_{i,2}}(t)=\varphi_{i,1}(t)X^*_2(t)+\varphi_{i,2}(t)\overline{X}^*_2(t)+\varphi_{i,0}(t),
\end{equation}
where $\varphi_{i,1}$, $\varphi_{i,2}$ are deterministic functions, $\varphi_{i,0}$ is an $\mathfrak{F}$-adapted process satisfying
$$
d\varphi_{i,0}(t)=\varphi_{i,3}(t)dt+\varphi_{i,4}(t)dB_t,\quad \varphi_0(T)=0.
$$
Taking conditional expectation with respect to $\mathfrak{F}_{\tau}$ in \eqref{1}, one has
$$
d\overline{X}^*_2=\left[(a_2+\overline{a}_2)\overline{X}^*_2+q^{-1}_{1,2}b_{1,2}^2\overline{P}_{v^*_{1,2}}+q^{-1}_{2,2}b_{2,2}^2\overline{P}_{v^*_{2,2}}\right]dt.
$$
Computing the differential of $P_{v^*_{i,2}}$, we have
\begin{align}\label{8}
dP_{v^*_{i,2}}&=\varphi'_{i,1}X^*_2dt+\varphi_{i,1}dX^*_2+\varphi'_{i,2}\overline{X}^*_2dt+\varphi_{i,2}d\overline{X}^*_2+\varphi'_{i,0}dt\nonumber\\
&=\left[\varphi'_{i,1}+\varphi_{i,1}\left(a_2+q^{-1}_{1,2}b_{1,2}^2\varphi_{1,1}+q^{-1}_{2,2}b_{2,2}^2\varphi_{2,1}\right)\right]X^*_2dt\nonumber
\\
&\quad\mbox{}+\Big\{\varphi'_{i,2}+\varphi_{i,1}\Big[\overline{a}_2+q^{-1}_{1,2}b_{1,2}^2\varphi_{1,2}+q^{-1}_{2,2}b_{2,2}^2\varphi_{2,2}\Big]\nonumber\\
&\quad\mbox{}+\varphi_{i,2}\Big[a_2+\overline{a}_2+q^{-1}_{1,2}b_{1,2}^2(\varphi_{1,1}+\varphi_{1,2})+q^{-1}_{2,2}b_{2,2}^2(\varphi_{2,1}+\varphi_{2,2})\Big]\Big\}\overline{X}^*_2dt
\nonumber\\
&\quad\mbox{}+\Big[\varphi_{i,3}+\varphi_{i,1}\left(q^{-1}_{1,2}b_{1,2}^2\varphi_{1,0}+q^{-1}_{2,2}b_{2,2}^2\varphi_{2,0}\right)+\varphi_{i,2}\left(q^{-1}_{1,2}b_{1,2}^2\overline{\varphi}_{1,0}+q^{-1}_{2,2}b_{2,2}^2\overline{\varphi}_{2,0}\right)\Big]dt\nonumber\\
&\quad\mbox{}+\left[\varphi_{i,1}\left(c_2X^*_2+\overline{c}_2\overline{X}^*_2\right)+\varphi_{i,4}\right]dB_t.
\end{align}
It follows from \eqref{8} and \eqref{1} that
\begin{equation*}
Q_{v^*_{i,2}}=\varphi_{i,1}\left(c_2X^*_2+\overline{c}_2\overline{X}^*_2\right)+\varphi_{i,4}.
\end{equation*}
Thus the second equation in \eqref{1} can be rewritten as
\begin{align}\label{11}
dP_{v^*_{i,2}}&=\left[p_{i,2}-a_2\varphi_{i,1}-\varphi_{i,1}c_2^2\right]X^*_2dt+\Big[\overline{p}_{i,2}-a_2\varphi_{i,2}-\bar{a}_2\left(\varphi_{i,1}+\varphi_{i,2}\right)
-\varphi_{i,1}\left(\overline{c}_2^2+2c_2\overline{c}_2\right)\Big]\overline{X}_2^*dt\nonumber\\
&\quad\mbox{}\;-\left(a_2\varphi_{i,0}+\overline{a}_2\overline{\varphi}_{i,0}+c_2\varphi_{i,4}+\overline{c}_2\overline{\varphi}_{i,4}\right)dt+\left[\varphi_{i,1}\left(c_2X^*_2+\overline{c}_2\overline{X}^*_2\right)+\varphi_{i,4}\right]dB_t.
\end{align}
Combining \eqref{8}, \eqref{11} and Assumption $(C)$, we can see that $(\varphi_{1,1},\varphi_{1,2},\varphi_{2,1},\varphi_{2,2})$ is a solution to the following system of ordinary differential equations (ODEs):
\begin{equation}\label{9}
\begin{cases}
\varphi'_{1,1}+(2a_2+c^2_2)\varphi_{1,1}+l_2\varphi^2_{1,1}+l_2\varphi_{1,1}\varphi_{2,1}-p_{1,2}=0, \quad t\in[\tau,T),\\
\varphi'_{2,1}+(2a_2+c^2_2)\varphi_{2,1}+l_2\varphi_{1,1}\varphi_{2,1}+l_2\varphi^2_{2,1}-p_{2,2}=0, \quad t\in[\tau,T),\\
\varphi'_{1,2}+\left(2a_2+2\overline{a}_2+2l_2\varphi_{1,1}+l_2\varphi_{2,1}\right)\varphi_{1,2}+l_2\varphi^2_{1,2}+l_2\varphi_{1,2}\varphi_{2,2}\\
\quad\;\;\,\mbox{}+l_2\varphi_{1,1}\varphi_{2,2}+\varphi_{1,1}\Big[2\overline{a}_2+\overline{c}_2^2+2c_2\overline{c}_2\Big]-\overline{p}_{1,2}=0,
\quad t\in[\tau,T),\\
\varphi'_{2,2}+\left(2a_2+2\overline{a}_2+2l_2\varphi_{2,1}+l_2\varphi_{1,1}\right)\varphi_{2,2}+l_2\varphi^2_{2,2}+l_2\varphi_{1,2}\varphi_{2,2}\\
\quad\;\;\,\mbox{}+l_2\varphi_{2,1}\varphi_{1,2}+\varphi_{2,1}\Big[2\overline{a}_2+\overline{c}_2^2+2c_2\overline{c}_2\Big]-\overline{p}_{2,2}=0,
\quad t\in[\tau,T),\\
\varphi_{1,1}(T)=-r_{1}(T),\quad \varphi_{1,2}(T)=0,\quad \varphi_{2,1}(T)=-r_{2}(T),\quad \varphi_{2,2}(T)=0.
\end{cases}
\end{equation}
Moreover, $((\varphi_{1,0},\varphi_{1,4}),(\varphi_{2,0},\varphi_{2,4}))$ is a solution to the following system of mean-field backward stochastic differential equations (MF-BSDEs):
\begin{equation}\label{10}
\begin{cases}
d\varphi_{1,0}=-\Big[\left(a_2+l_2\varphi_{1,1}\right)\varphi_{1,0}+\left(\overline{a}_2+l_2\varphi_{1,2}\right)\overline{\varphi}_{1,0}+c_2\varphi_{1,4}+\overline{c}_2\overline{\varphi}_{1,4}\\
\qquad\quad\,\,\mbox{}+l_2\left(\varphi_{1,1}\varphi_{2,0}+\varphi_{1,2}\overline{\varphi}_{2,0}\right)\Big]dt+\varphi_{1,4}dB_t,
\quad t\in[\tau,T),\\
d\varphi_{2,0}=-\Big[\left(a_2+l_2\varphi_{2,1}\right)\varphi_{2,0}+\left(\overline{a}_2+l_2\varphi_{2,2}\right)\overline{\varphi}_{2,0}+c_2\varphi_{2,4}+\overline{c}_2\overline{\varphi}_{2,4}\\
\qquad\quad\,\,\mbox{}+l_2\left(\varphi_{2,1}\varphi_{1,0}+\varphi_{2,2}\overline{\varphi}_{1,0}\right)\Big]dt+\varphi_{2,4}dB_t,
\quad t\in[\tau,T),\\
\varphi_{1,0}(T)=\varphi_{2,0}(T)=0.
\end{cases}
\end{equation}

Now we show the existence and uniqueness of solutions to \eqref{9} and \eqref{10}, respectively.

\begin{lemma}
There exists a unique solution $(\varphi_{1,1},\varphi_{1,2},\varphi_{2,1},\varphi_{2,2})$ to \eqref{9}.
\end{lemma}

\begin{proof}
Set $\varphi_{j}=\varphi_{1,j}+\varphi_{2,j}$. Then, $\varphi_{1}$ satisfies the following Riccati equation:
\begin{equation}\label{12}
\begin{cases}
\varphi'_{1}+(2a_2+c^2_2)\varphi_{1}+l_2\varphi^2_{1}-(p_{1,2}+p_{2,2})=0, \quad t\in[\tau,T),\\
\varphi_1(T)=-r_{1}(T)-r_{2}(T).
\end{cases}
\end{equation}
By Proposition 7.1 in \cite{Yong2009Stochastic}, there exists a unique solution to \eqref{12}. Therefore, \eqref{9} can be transformed into the following  system of  linear equations:
\begin{equation*}
\begin{cases}
\varphi'_{1,1}+(2a_2+c^2_2)\varphi_{1,1}+l_2\varphi_{1,1}\varphi_{1}-p_{1,2}=0, \quad t\in[\tau,T),\\
\varphi'_{2,1}+(2a_2+c^2_2)\varphi_{2,1}+l_2\varphi_{2,1}\varphi_{1}-p_{2,2}=0, \quad t\in[\tau,T),\\
\varphi_{1,1}(T)=-r_{1}(T),\quad \varphi_{2,1}(T)=-r_{2}(T),
\end{cases}
\end{equation*}
where $\varphi_{1}$ is the unique solution to \eqref{12}. The existence and uniqueness of solutions $\varphi_{1,1}$ and $\varphi_{2,1}$ is obtained immediately.

Similarly, $\varphi_{2}$ satisfies the following Riccati equation:
\begin{equation*}
\begin{cases}
\varphi'_{2}+\left(2a_2+2\overline{a}_2+2l_2\varphi_{1}\right)\varphi_{2}+l_2\varphi^2_{2}+\varphi_{1}\Big[2\overline{a}_2+\overline{c}_2^2+2c_2\overline{c}_2\Big]-(\overline{p}_{1,2}+\overline{p}_{2,2})=0, \quad t\in[\tau,T),\\
\varphi_2(T)=0.
\end{cases}
\end{equation*}
Repeating the above arguments, we can obtain the existence and uniqueness of solutions $\varphi_{2,1}$ and $\varphi_{2,2}$.
\end{proof}

\begin{lemma}
There exists a unique $\mathfrak{F}_t$-adapted solution $((\varphi_{1,0},\varphi_{1,4}),(\varphi_{2,0},\varphi_{2,4}))$ to \eqref{10}.
\end{lemma}

\begin{proof}
Set
\begin{align*}
&\alpha_2=
\begin{bmatrix}
\varphi_{1,0}\\
\varphi_{2,0}
\end{bmatrix},\quad
\beta_2=
\begin{bmatrix}
\varphi_{1,4}\\
\varphi_{2,4}
\end{bmatrix},\quad
\mu_2=
\begin{bmatrix}
a_2+l_2\varphi_{1,1} &l_2\varphi_{1,1}\\
l_2\varphi_{2,1} &a_2+l_2\varphi_{2,1}
\end{bmatrix},\\
&\overline{\alpha}_2=
\begin{bmatrix}
\overline{\varphi}_{1,0}\\
\overline{\varphi}_{2,0}
\end{bmatrix},\quad
\overline{\beta}_2=
\begin{bmatrix}
\overline{\varphi}_{1,4}\\
\overline{\varphi}_{2,4}
\end{bmatrix},\quad
\nu_2=
\begin{bmatrix}
\overline{a}_2+l_2\varphi_{1,2} &l_2\varphi_{1,2}\\
l_2\varphi_{2,2} &\overline{a}_2+l_2\varphi_{2,2}
\end{bmatrix}.
\end{align*}
Then \eqref{10} can be rewritten as
\begin{equation*}
\begin{cases}
d\alpha_2=-\left(\mu_2\alpha_2+\nu_2\overline{\alpha}_2+c_2\beta+\overline{c}_2\overline{\beta}_2\right)+\beta_2 dB_t,
\quad t\in[\tau,T),\\
\alpha_2(T)=0,
\end{cases}
\end{equation*}
which is a linear 2-dimensional BSDE. Thus the result is obtained directly.
\end{proof}

Next, we aim to solve the optimal state $X^*_2$ at the second stage. Combining \eqref{1} and \eqref{3} derives the following MF-SDE:
\begin{equation}\label{15}
\begin{cases}
dX^*_2=\left[(a_2+l_2\varphi_1)X^*_2+(\overline{a}_2+l_2\varphi_2)\overline{X}^*_2+l_2\varphi_0\right]dt+\left[c_2X^*_2+\overline{c}_2\overline{X}^*_2\right]dB_t,\quad t\in[\tau,T),\\
X^*_2(\tau)=X_1(\tau).
\end{cases}
\end{equation}
Observing that \eqref{15} is linear, we set
\begin{equation}\label{17}
X^*_2(t)=M(t)X^*_2(\tau)+N(t).
\end{equation}
Substituting \eqref{17} into \eqref{15} yields
\begin{align*}
X^*_2(t)&=X^*_2(\tau)+X^*_2(\tau)\int_{\tau}^t(a_2(s)+l_2(s)\varphi_1(s))M(s)+(\overline{a}_2(s)+l_2(s)\varphi_2(s))\overline{M}(s)ds\\
&\quad\mbox{}+\int_{\tau}^t(a_2(s)+l_2(s)\varphi_1(s))N(s)+(\overline{a}_2(s)+l_2(s)\varphi_2(s))\overline{N}(s)+l_2(s)\varphi_0(s)ds\\
&\quad\mbox{}+X^*_2(\tau)\int_{\tau}^tc_2(s)M(s)+\overline{c}_2(s)\overline{M}(s)dB_s+\int_{\tau}^tc_2(s)N(s)+\overline{c}_2(s)\overline{N}(s)dB_s.
\end{align*}
Thus, $M(t)$ and $N(t)$ are solutions to the following  system of linear MF-SDEs:
\begin{equation}\label{18}
\begin{cases}
dM=\left[(a_2+l_2\varphi_1)M+(\overline{a}_2+l_2\varphi_2)\overline{M}\right]dt+\left[c_2M+\overline{c}_2\overline{M}\right]dB_t,\quad t\in(\tau,T],\\
dN=\left[(a_2+l_2\varphi_1)N+(\overline{a}_2+l_2\varphi_2)\overline{N}+l_2\varphi_0\right]dt+\left[c_2N+\overline{c}_2\overline{N}\right]dB_t,\quad t\in(\tau,T],\\
M(\tau)=1, \quad N(\tau)=0.
\end{cases}
\end{equation}

We also need the following lemma.
\begin{lemma}\label{40}
At $(v_{1,2},v_{2,2})=(v^*_{1,2},v^*_{2,2})$, the value of the objective subfunctions for followers is
\begin{equation}\label{19}
J_{i,2}(v^*_{1,2},v^*_{2,2})=\frac{1}{2}\left[\left((\varphi_{i,1}+\varphi_{i,2})X^{*2}_2
+\varphi_{i,0}X^*_2\right)(\tau)+\mathbb{E}\Big[\int_{\tau}^Tl_2(t)P_{v^*_{1,2}}(t)P_{v^*_{2,2}}(t)dt\Big|\mathfrak{F}_{\tau}\right].
\end{equation}
\end{lemma}

\begin{proof}
Applying It\^{o}'s formula to $X^*_2P_{v^*_{i,2}}$ and taking conditional expectation with respect to $\mathfrak{F}_{\tau}$, one has
\begin{align*}
&\quad\; \mathbb{E}\left[X^*_2(T)P_{v^*_{1,2}}(T)-X^*_2(\tau)P_{v^*_{i,2}}(\tau)\Big|\mathfrak{F}_{\tau}\right]\\
&= -\mathbb{E}\left[r_{i}(T)X^{*2}_2(T)\Big|\mathfrak{F}_{\tau}\right]-X^*_2(\tau)P_{v^*_{i,2}}(\tau)\\
&= \mathbb{E}\Big[\int_{\tau}^T
\left[q^{-1}_{1,2}b_{1,2}^2P_{v^*_{1,2}}P_{v^*_{i,2}}+q^{-1}_{2,2}b_{2,2}^2P_{v^*_{2,2}}P_{v^*_{i,2}}+P_{v^*_{1,2}}X^{*2}_2
+\overline{P}_{v^*_{1,2}}\overline{X}^{*}_2{X}^{*}_2\right](t)dt\Big|\mathfrak{F}_{\tau}\Big].
\end{align*}

For follower $\mathcal{P}_1$, it follows from the fact $q^{-1}_{1,2}b_{1,2}^2P_{v^*_{1,2}}P_{v^*_{1,2}}=q_{1,2}v^{*2}_{1,2}$ that
\begin{align*}
&\quad\;X^*_2(\tau)P_{v^*_{1,2}}(\tau)\\
&=\varphi_{1,1}(\tau)X^{*2}_2(\tau)+\varphi_{1,2}(\tau)\overline{X}^*_2(\tau)X^*_2(\tau)+\varphi_{1,0}(\tau)X^*_2(\tau)\\
&=(\varphi_{1,1}(\tau)+\varphi_{1,2}(\tau))X^{*2}_2(\tau)+\varphi_{1,0}(\tau)X^*_2(\tau)\\
&=-\mathbb{E}\Big[\int_{\tau}^T\left[q_{1,2}v^{*2}_{1,2}+q^{-1}_{2,2}b_{2,2}^2P_{v^*_{2,2}}P_{v^*_{1,2}}+P_{v^*_{1,2}}X^{*2}_2
+\overline{P}_{v^*_{1,2}}\overline{X}^{*2}_2\right](t)dt+r_{i}(T)X^{*2}_2(T)\Big|\mathfrak{F}_{\tau}\Big].
\end{align*}
This implies
$$
J_{1,2}(v^*_{1,2},v^*_{2,2})=\frac{1}{2}\left[(\varphi_{1,1}(\tau)+\varphi_{1,2}(\tau))X^{*2}_2(\tau)+\varphi_{1,0}(\tau)X^*_2(\tau)
+\mathbb{E}\Big[\int_{\tau}^Tl_2(t)P_{v^*_{2,2}}(t)P_{v^*_{1,2}}(t)dt\Big|\mathfrak{F}_{\tau}\right].
$$
Analogously, for follower $\mathcal{P}_2$, one has
$$
J_{2,2}(v^*_{1,2},v^*_{2,2})=\frac{1}{2}\left[(\varphi_{2,1}(\tau)+\varphi_{2,2}(\tau))X^{*2}_2(\tau)+\varphi_{2,0}(\tau)X^*_2(\tau)+
\mathbb{E}\Big[\int_{\tau}^Tl_2(t)P_{v^*_{1,2}}(t)P_{v^*_{2,2}}(t)dt\Big|\mathfrak{F}_{\tau}\right].
$$
This ends the proof.
\end{proof}

Now, we are able to give the formulation of the optimal objective subfunctions at the second stage. Combining \eqref{3}, \eqref{17} and Lemma \ref{40}, one has
\begin{align}\label{20}
\mathbb{E}&\left[\int_{\tau}^Tl_2(t)P_{v^*_{1,2}}(t)P_{v^*_{2,2}}(t)dt\Big|\mathfrak{F}_{\tau}\right]\nonumber\\
=\mathbb{E}&\Big[\int_{\tau}^Tl_2\left[\left(\varphi_{1,1}M+\varphi_{1,2}\overline{M}\right)X^*_2(\tau)+\left(\varphi_{1,1}N+\varphi_{1,2}\overline{N}+\varphi_{1,0}\right)\right]\nonumber\\
&\times\left[\left(\varphi_{2,1}M+\varphi_{2,2}\overline{M}\right)X^*_2(\tau)+\left(\varphi_{2,1}N+\varphi_{2,2}\overline{N}+\varphi_{2,0}\right)\right](t)dt\Big|\mathfrak{F}_{\tau}\Big]\nonumber\\
=\mathbb{E}&\Big[\int_{\tau}^Tl_2\left(\varphi_{1,1}M+\varphi_{1,2}\overline{M}\right)\left(\varphi_{2,1}M+\varphi_{2,2}\overline{M}\right)(t)dt\Big|\mathfrak{F}_{\tau}\Big]\left(X^{*}_2(\tau)\right)^2+\mathbb{E}\Big[\int_{\tau}^Tl_2\Big[\left(\varphi_{1,1}M+\varphi_{1,2}\overline{M}\right)\nonumber\\
&\left(\varphi_{2,1}N+\varphi_{2,2}\overline{N}+\varphi_{2,0}\right)+\left(\varphi_{1,1}N+\varphi_{1,2}\overline{N}+\varphi_{1,0}\right)\left(\varphi_{2,1}M+\varphi_{2,2}\overline{M}\right)\Big](t)\Big|\mathfrak{F}_{\tau}\Big]X^{*}_2(\tau)\nonumber\\
\mbox{}& + \mathbb{E}\Big[\int_{\tau}^Tl_2\left(\varphi_{2,1}N+\varphi_{2,2}\overline{N}+\varphi_{2,0}\right)\left(\varphi_{1,1}N+\varphi_{1,2}\overline{N}+\varphi_{1,0}\right)(t)dt\Big|\mathfrak{F}_{\tau}\Big].
\end{align}
Substituting \eqref{20} into \eqref{19}, we have the following theorem.

\begin{thm}\label{24}
For player $\mathcal{P}_i$, the optimal objective subfunction is
$$
J_{i,2}(v^*_{1,2},v^*_{2,2})=\frac{1}{2}\left[\kappa_{i,1}(\tau)X^{*2}_2(\tau)
+\kappa_{i,2}(\tau)X^*_2(\tau)+\kappa_{i,0}(\tau)\right],
$$
where
\begin{equation*}
\begin{cases}
\kappa_{i,1}(\tau)=\varphi_{i,1}(\tau)+\varphi_{i,2}(\tau)+\mathbb{E}\Big[\int_{\tau}^Tl_2\left(\varphi_{1,1}M+\varphi_{1,2}\overline{M}\right)\left(\varphi_{2,1}M+\varphi_{2,2}\overline{M}\right)(t)dt\Big|\mathfrak{F}_{\tau}\Big],\\
\kappa_{i,2}(\tau)=\varphi_{i,0}(\tau)+\mathbb{E}\Big[\int_{\tau}^Tl_2\Big[\left(\varphi_{1,1}M+\varphi_{1,2}\overline{M}\right)\left(\varphi_{2,1}N+\varphi_{2,2}\overline{N}+\varphi_{2,0}\right)\\
\qquad\quad\;\mbox{}+\left(\varphi_{1,1}N+\varphi_{1,2}\overline{N}+\varphi_{1,0}\right)\left(\varphi_{2,1}M+\varphi_{2,2}\overline{M}\right)\Big](t)\Big|\mathfrak{F}_{\tau}\Big],\\
\kappa_{i,0}(\tau)=\mathbb{E}\Big[\int_{\tau}^Tl_2\left(\varphi_{2,1}N+\varphi_{2,2}\overline{N}+\varphi_{2,0}\right)\left(\varphi_{1,1}N+\varphi_{1,2}\overline{N}+\varphi_{1,0}\right)(t)dt\Big|\mathfrak{F}_{\tau}\Big].
\end{cases}
\end{equation*}
\end{thm}

\begin{remark}\label{5}
By Theorem 3.8 Chapter IX in \cite{Revuz1999Continuous}, it is easy to show that $\varphi_{i,1}(t)<0$, $\varphi_{i,2}(t)<0$ and $M(t)>0$. Therefore, $\kappa_{i,1}(\tau)<0$ a.s. and the concavity of $J_{i,2}(v^*_{1,2},v^*_{2,2})$ with respect to $X^*_2(\tau)$ follows.
\end{remark}

\subsection{Nash equilibrium for followers at the first stage}
In this subsection, we solve the Nash equilibrium point $(v^*_{1,1},v^*_{2,1})$ for followers at the first stage. According to Theorem \ref{24}, Subproblem \eqref{23} can be rewritten as follows:
\begin{subprob}\label{42}
Find $(v^*_{1,1},v^*_{2,1})\in \mathcal{V}[0,\tau]\times \mathcal{V}[0,\tau]$ such that
\begin{equation*}
\begin{cases}
J_{1,1}(v^*_{1,1},v^*_{2,1},v_0)
=\mathop{\esssup}  \limits_{v_{1,1} \in \mathcal{V}[0,\tau]}\mathbb{E}\Big\{\int_0^{\tau}-\frac{1}{2}\left[p_{1,1}(t)X_1^2(t)+\overline{p}_{1,1}(t)\overline{X}_2^2(t)+q_{1,1}(t)v_{1,1}^2(t)\right]dt\\
\qquad\qquad\qquad\qquad\qquad\qquad\qquad\;\;\mbox{}+\frac{1}{2}\left[\kappa_{1,1}(\tau)X^{*2}_1(\tau)
+\kappa_{1,2}(\tau)X^*_1(\tau)+\kappa_{1,0}(\tau)\right]\Big|\mathfrak{F}_{0} \Big\},\\
J_{2,1}(v^*_{1,1},v^*_{2,1},v_0)
=\mathop{\esssup}  \limits_{v_{2,1} \in \mathcal{V}[0,\tau]}\mathbb{E}\Big\{\int_0^{\tau}-\frac{1}{2}\left[p_{2,1}(t)X_1^2(t)+\overline{p}_{2,1}(t)\overline{X}_2^2(t)+q_{2,1}(t)v_{2,1}^2(t)\right]dt\\
\qquad\qquad\qquad\qquad\qquad\qquad\qquad\;\;\mbox{}+\frac{1}{2}\left[\kappa_{2,1}(\tau)X^{*2}_1(\tau)
+\kappa_{2,2}(\tau)X^*_1(\tau)+\kappa_{2,0}(\tau)\right]\Big|\mathfrak{F}_{0} \Big\}.
\end{cases}
\end{equation*}
\end{subprob}

Similar to Subsection 3.1, we can define the Hamiltonian function of follower $\mathcal{P}_i$ at first stage by
\begin{align*}
H_{i,1}&:[0,\tau]\times \mathbb{R}\times \mathbb{R}\times\mathbb{R}\times \mathbb{R}\times \mathbb{R}\times \mathbb{R}\times \mathbb{R}\rightarrow \mathbb{R},\\
H_{i,1}&=H_{i,1}(t,X_1,\overline{X}_1,v_0,v_{1,1},v_{2,1},P_{v_{i,1}},Q_{v_{i,1}}),\\
&=-\frac{1}{2}\left[p_{i,1}X_1^2+\overline{p}_{i,1}\overline{X}_1^2+q_{i,1}v_{i,1}^2\right]+\left[a_1X_1+\overline{a}_1\overline{X}_1+b_{1,1}v_{1,1}+b_{2,1}v_{2,1}+b_0v_0\right]P_{v_{i,1}}\\
&\quad\,\mbox{}+\left[c_1X_1+\overline{c}_1\overline{X}_1\right]Q_{v_{i,1}}.
\end{align*}
Recalling Remark \ref{5} and using the sufficient and necessary maximum principles of MF-SDEs with default (see \cite{gou2020optimal}), we have
$$
0=\frac{\partial H_{i,1}}{\partial v_{i,1}}\Big|_{v_{i,1}=v^*_{i,1}}=-q_{i,1}v^*_{i,1}+b_{i,1}P_{v^*_{i,1}}\quad \Rightarrow \quad v^*_{i,1}=q^{-1}_{i,1}b_{i,1}P_{v^*_{i,1}},
$$
where $P_{v^*_{i,1}}$ is the solution to corresponding linear BSDE.

Now we summarize the above arguments as the following theorem.

\begin{thm}\label{13}
$(v^*_{1,1},v^*_{2,1})$ is a Nash equilibrium point for Subproblem \ref{42} if and only if
\begin{equation}\label{81}
(v^*_{1,1},v^*_{2,1})=(q^{-1}_{1,1}b_{i,1}P_{v^*_{1,1}},q^{-1}_{2,1}b_{2,1}P_{v^*_{2,1}}),
\end{equation}
where the optimal triple $(X_1^*,P_{v^*_{i,1}},Q_{v^*_{i,1}})$ enjoys the following MF-FBSDE:
\begin{equation}\label{82}
\begin{cases}
dX^*_1=\left[a_1X^*_1+\overline{a}_1\overline{X}^*_1+q^{-1}_{1,1}b_{1,1}^2P_{v^*_{1,1}}+q^{-1}_{2,1}b_{2,1}^2P_{v^*_{2,1}}+b_0v_0\right]dt+\left[c_1X^*_1+\overline{c}_1\overline{X}^*_1\right]dB_t,\quad t\in(0,\tau],\\
dP_{v^*_{i,1}}=\left[p_{i,1}X^*_1+\overline{p}_{i,1}\overline{X}^*_1-a_1P_{v^*_{i,1}}-\overline{a}_1\overline{P}_{v^*_{i,1}}-c_1Q_{v^*_{i,1}}-\overline{c}_1\overline{Q}_{v^*_{i,1}}\right]dt+Q_{v^*_{i,1}}dB_t, \quad t\in[0,\tau),\\
X_1^*(\tau)=x_0,\quad P_{v^*_{i,1}}(\tau)=\kappa_{i,1}(\tau)X^*_2(\tau)+\kappa_{i,2}(\tau).
\end{cases}
\end{equation}
\end{thm}

Next, we put
\begin{equation}\label{27}
P_{v^*_{i,1}}(t)=\psi_{i,1}(t)X^*_1(t)+\psi_{i,2}(t)\overline{X}^*_1(t)+\psi_{i,0}(t),
\end{equation}
where $\psi_{i,1}$, $\psi_{i,2}$ are deterministic functions, $\psi_{i,0}$ is an $\mathfrak{F}$-adapted process such that
$$
d\psi_{i,0}(t)=\psi_{i,3}(t)dt+\psi_{i,4}(t)dB_t,\; \psi_0(T)=0.
$$
Moreover, $\psi_{j}=\psi_{1,j}+\psi_{2,j}$.

Then, repeating the arguments in Subsection 3.1, we can obtain the following system of ODEs:
\begin{equation}\label{28}
\begin{cases}
\psi'_{1,1}+(2a_1+c^1_2)\psi_{1,1}+l_1\psi^2_{1,1}+l_1\psi_{1,1}\psi_{2,1}-p_{1,1}=0, \quad t\in[0,\tau),\\
\psi'_{2,1}+(2a_1+c^1_2)\psi_{2,1}+l_1\psi_{1,1}\psi_{2,1}+l_1\psi^2_{2,1}-p_{2,1}=0, \quad t\in[0,\tau),\\
\psi'_{1,2}+\left(2a_1+2\overline{a}_1+2l_1\psi_{1,1}+l_1\psi_{2,1}\right)\psi_{1,2}+l_1\psi^2_{1,2}+l_1\psi_{1,2}\psi_{2,2}\\
\quad\;\;\,\mbox{}+l_1\psi_{1,1}\psi_{2,2}+\psi_{1,1}\Big[2\overline{a}_1+\overline{c}_1^2+2c_1\overline{c}_1\Big]-\overline{p}_{1,2}=0,
\quad t\in[0,\tau),\\
\psi'_{2,2}+\left(2a_1+2\overline{a}_1+2l_1\psi_{2,1}+l_1\psi_{1,1}\right)\psi_{2,2}+l_1\psi^2_{2,2}+l_1\psi_{1,2}\psi_{2,2}\\
\quad\;\;\,\mbox{}+l_1\psi_{2,1}\psi_{1,2}+\psi_{2,1}\Big[2\overline{a}_1+\overline{c}_1^2+2c_1\overline{c}_1\Big]-\overline{p}_{2,2}=0,
\quad t\in[0,\tau),\\
\psi_{1,1}(\tau)=\kappa_{1,1}(\tau),\quad \psi_{1,2}(\tau)=0,\quad \psi_{2,1}(\tau)=\kappa_{2,1}(\tau),\quad \psi_{2,2}(\tau)=0
\end{cases}
\end{equation}
for $(\psi_{1,1},\psi_{1,2},\psi_{2,1},\psi_{2,2})$, and the following system of MF-BSDEs:
\begin{equation}\label{29}
\begin{cases}
d\psi_{1,0}=-\Big[\psi_{1,1}b_0v_0+\psi_{1,2}b_0\overline{v}_0+\left(a_1+l_1\psi_{1,1}\right)\psi_{1,0}+\left(\overline{a}_1+l_1\psi_{1,2}\right)\overline{\psi}_{1,0}+c_1\psi_{1,4}+\overline{c}_1\overline{\psi}_{1,4}\\
\qquad\quad\,\,\mbox{}+l_1\left(\psi_{1,1}\psi_{2,0}+\psi_{1,2}\overline{\psi}_{2,0}\right)\Big]dt+\psi_{1,4}dB_t,
\quad t\in[0,\tau),\\
d\psi_{2,0}=-\Big[\psi_{2,1}b_0v_0+\psi_{2,2}b_0\overline{v}_0+\left(a_1+l_1\psi_{2,1}\right)\psi_{2,0}+\left(\overline{a}_1+l_1\psi_{2,2}\right)\overline{\psi}_{2,0}+c_1\psi_{2,4}+\overline{c}_1\overline{\psi}_{2,4}\\
\qquad\quad\,\,\mbox{}+l_1\left(\psi_{2,1}\psi_{1,0}+\psi_{2,2}\overline{\psi}_{1,0}\right)\Big]dt+\psi_{2,4}dB_t,
\quad t\in[0,\tau),\\
\psi_{1,0}(\tau)=\kappa_{1,2}(\tau),\quad \psi_{2,0}(\tau)=\kappa_{2,2}(\tau)
\end{cases}
\end{equation}
for $((\psi_{1,0},\psi_{1,4}),(\psi_{2,0},\psi_{2,4}))$. Analogously, we know that both \eqref{28} and \eqref{29} have unique solutions.
By setting
\begin{align*}
&\alpha_1=
\begin{bmatrix}
\psi_{1,0}\\
\psi_{2,0}
\end{bmatrix},\quad
&&\beta_1=
\begin{bmatrix}
\psi_{1,4}\\
\psi_{2,4}
\end{bmatrix},\quad
&&\mu_1=
\begin{bmatrix}
a_1+l_1\psi_{1,1} &l_1\psi_{1,1}\\
l_1\psi_{2,1} &a_1+l_1\psi_{2,1}
\end{bmatrix},\\
&\overline{\alpha}_1=
\begin{bmatrix}
\overline{\psi}_{1,0}\\
\overline{\psi}_{2,0}
\end{bmatrix},\quad
&&\overline{\beta}_1=
\begin{bmatrix}
\overline{\psi}_{1,4}\\
\overline{\psi}_{2,4}
\end{bmatrix},\quad
&&\nu_1=
\begin{bmatrix}
\overline{a}_1+l_1\psi_{1,2} &l_1\psi_{1,2}\\
l_1\psi_{2,2} &\overline{a}_1+l_1\psi_{2,2}
\end{bmatrix},\\
&\delta_1=
\begin{bmatrix}
\psi_{1,1}b_0\\
\psi_{2,1}b_0
\end{bmatrix},\quad
&&\overline{\delta}_1=
\begin{bmatrix}
\psi_{1,2}b_0\\
\psi_{2,2}b_0
\end{bmatrix},\quad
&&\iota(\tau)=\begin{bmatrix}
\kappa_{1,2}(\tau)\\
\kappa_{2,2}(\tau)
\end{bmatrix},
\end{align*}
\eqref{29} can be rewritten as
\begin{equation}\label{4}
\begin{cases}
d\alpha_1=-\left(\mu_1\alpha_1+\nu_1\overline{\alpha}_1+c_1\beta_1+\overline{c}_1\overline{\beta}_1+\delta_1v_0+\overline{\delta}_1\overline{v}_0\right)+\beta_1 dB_t,
\quad t\in[0,\tau),\\
\alpha_1(\tau)=\iota(\tau).
\end{cases}
\end{equation}
Thus, at the first stage, the optimal state $X^*_1(t)$ at Nash equilibrium point $(v^*_{1,1},v^*_{2,1})$ is uniquely determined by the following system:
\begin{equation}\label{36}
\begin{cases}
dX^*_1=\left[(a_1+l_1\psi_1)X^*_1+(\overline{a}_1+l_1\psi_2)\overline{X}^*_1+(l_1,l_1)\alpha_1+b_0v_{0}\right]dt\\
\qquad\quad\,\mbox{}+\left[c_1X^*_1+\overline{c}_1\overline{X}^*_1\right]dB_t, \quad t\in(0,\tau],\\
d\alpha_1=-\left(\mu_1\alpha_1+\nu_1\overline{\alpha}_1+c_1\beta_1+\overline{c}_1\overline{\beta}_1+\delta_1v_0+\overline{\delta}_1\overline{v}_0\right)+\beta_1 dB_t,
\quad t\in[0,\tau),\\
X^*_1(0)=x_0,\quad \alpha_1(\tau)=\iota(\tau).
\end{cases}
\end{equation}

\subsection{Stackelberg solution of LQ-MF-SSDG with random exit time}
In this subsection, we focus on finding the optimal strategy $v^*_0$ for the leader and then deriving optimal strategies of followers by setting $v_{0}=v^*_{0}$ in \eqref{82}. Combining \eqref{31}, \eqref{37}, \eqref{38} and \eqref{36}, we obtain the following optimal strategy problem for the leader, which is indeed an optimal control problem for MF-SDEs with default.
\begin{subprob}\label{88}
Find $v^*_{0} \in \mathcal{V}[0,\tau]$ such that
\begin{align*}
&\quad\;J_0(v^*_1,v^*_2,v^*_0)=\mathop{\esssup}\limits_{v_{0} \in \mathcal{V}[0,\tau]}J_0(v^*_1,v^*_2,v_0)\\
&=\mathop{\esssup}\limits_{v_{0}\in \mathcal{V}[0,\tau]}\left\{-\frac{1}{2}\mathbb{E}\left[\int_0^{\tau}p_0(t)Y^2(t)+\overline{p}_0(t)\overline{Y}^2(t)
+q_0(t)v_0^2(t)dt+r_0(\tau)Y^2(\tau)\Big|\mathfrak{F}_{0}\right]\right\},
\end{align*}
where the state process $Y(t)$ is governed by
\begin{equation*}
\begin{cases}
dY=\left[(a_1+l_1\psi_1+\gamma^{\mathfrak{F}}d_0)Y+(\overline{a}_1+l_1\psi_2+\gamma^{\mathfrak{F}}\overline{d}_0)\overline{Y}+(l_1,l_1)\alpha_1+b_0v_{0}\right]dt\\
\qquad\;\,\:\mbox{}+\left[c_1Y+\overline{c}_1\overline{Y}\right]dB_t+\left[d_0Y+\overline{d}_0\overline{Y}\right]dA_t, \quad t\in(0,\tau],\\
d\alpha_1=-\left(\mu_1\alpha_1+\nu_1\overline{\alpha}_1+c_1\beta_1+\overline{c}_1\overline{\beta}_1+\delta_1v_0+\overline{\delta}_1\overline{v}_0\right)dt+\beta_1 dB_t,
\quad t\in[0,\tau);\\
Y(0)=x_0,\quad \alpha_1(\tau)=\iota(\tau).
\end{cases}
\end{equation*}
\end{subprob}
To solve Subproblem \ref{88}, we define the Hamiltonian function of the leader by setting
\begin{align*}
H_{0}&:[0,\tau]\times \mathbb{R}\times \mathbb{R}\times\mathbb{R}\times \mathbb{R}^2\times \mathbb{R}\times\mathbb{R}\times \mathbb{R}\times \mathbb{R}^2\rightarrow \mathbb{R},\\
H_{0}&=H_{0}(t,Y,\overline{Y},v_0,\alpha_1,P_{v_{0}},Q_{v_{0}},R_{v_{0}},Z_{v_{0}})\\
&=-\frac{1}{2}\left[p_{0}Y^2+\overline{p}_{0}\overline{Y}^2+q_{0}v_{0}^2\right]+\left\langle P_{v_{0}},(a_1+l_1\psi_1+\gamma^{\mathfrak{F}}d_0)Y+(\overline{a}_1+l_1\psi_2+\gamma^{\mathfrak{F}}\overline{d}_0)\overline{Y}+(l_1,l_1)\alpha_1+b_0v_{0}\right\rangle\\
&\quad\,\mbox{}+\left\langle Q_{v_{0}},c_1Y+\overline{c}_1\overline{Y}\right\rangle+\left\langle R_{v_{0}},(d_0Y+\overline{d}_0\overline{Y})\gamma^{\mathcal{F}}\right\rangle
+\left\langle Z_{v_{0}},\mu_1\alpha_1+\nu_1\overline{\alpha}_1+c_1\beta_1+\overline{c}_1\overline{\beta}_1+\delta_1v_0+\overline{\delta}_1\overline{v}_0\right\rangle,
\end{align*}
where $(P_{v_{0}},Q_{v_{0}},R_{v_{0}},Z_{v_{0}})$ is the  solution to the following MF-FBSDE:
\begin{equation*}
\begin{cases}
dP_{v_{0}}=\Big\{p_{0}Y+\overline{p}_{0}\overline{Y}-(a_1+l_1\psi_1+\gamma^{\mathfrak{F}}d_0)P_{v_{0}}-(\overline{a}_1+l_1\psi_2)\overline{P}_{v_{0}}-\overline{d}_0\mathbb{E}\left[\gamma^{\mathfrak{F}}P_{v_{0}}\Big|\mathfrak{F}_{0}\right]\\
\qquad\;\,\:\mbox{}-(c_1Q_{v_{0}}+\overline{c}_1\overline{Q}_{v_{0}})-(R_{v_{0}}\gamma^{\mathfrak{F}}d_0+\overline{d}_0\mathbb{E}\left[\gamma^{\mathfrak{F}}R_{v_{0}}\Big|\mathfrak{F}_{0}\right])\Big\}dt+Q_{v_{0}}dB_t+R_{v_{0}}dA_t, \quad t\in(0,\tau],\\
dZ_{v_{0}}=\left(
\begin{bmatrix}
l_1\\
l_1
\end{bmatrix}
P_{v_{0}}+\mu_1^TZ_{v_{0}}+\nu_1^T\overline{Z}_{v_{0}}\right)dt+\left(c_1Z_{v_{0}}+\overline{c}_1\overline{Z}_{v_{0}}\right) dB_t,
\quad t\in[0,\tau),\\
P_{v_{0}}(\tau)=-r_0(\tau)Y(\tau),\quad Z_{v_{0}}(0)=
\begin{bmatrix}
0\\
0
\end{bmatrix}.
\end{cases}
\end{equation*}
Here $\mu_1^T$ and $\nu_1^T$ represent the transposed matrixes of $\mu_1$ and $\nu_1$, respectively. It is easy to show that
$v_0$ becomes an optimal control of Subproblem \ref{42} when
$$
0=\frac{\partial H_{0}}{\partial v_{0}}+\mathbb{E}\left[\nabla_{\overline{v}_0}H_0\Big|\mathfrak{F}_{0}\right]\Big|_{v_{0}=v^*_{0}}
=-q_{0}v^*_{0}+b_0P_{v^*_{0}}+\langle\delta_1,Z_{v^*_{0}}\rangle+\langle\overline{\delta}_1,\overline{Z}_{v^*_{0}}\rangle,
$$
where $\nabla_{\overline{v}_0}$ is the Fr\'{e}chet derivative with respect to $\overline{v}_0$. It remains to verify that, for
\begin{equation}\label{50}
v_0^*=q_0^{-1}(b_0P_{v^*_{0}}+\langle\delta_1,Z_{v^*_{0}}\rangle+\langle\overline{\delta}_1,\overline{Z}_{v^*_{0}}\rangle),
\end{equation}
the following inequality holds:
$$
J_0(v^*_1,v^*_2,v^*_0)\leq J_0(v^*_1,v^*_2,v_0), \quad \forall v_0\in\mathcal{V}[0,\tau].
$$
Setting $\widetilde{Y}=Y^*-Y$, $\widetilde{\alpha}_1=\alpha_1^*-\alpha_1$ and $\widetilde{v}_0=v_0^*-v_0$, one has
\begin{align}\label{51}
&\quad\;J_0(v^*_1,v^*_2,v^*_0)-J_0(v^*_1,v^*_2,v_0)\nonumber\\
&=-\frac{1}{2}\mathbb{E}\left[\int_0^{\tau}p_0(t)\widetilde{Y}^{*2}(t)+\overline{p}_0(t)\widetilde{\overline{Y}}^{*2}(t)+q_0(t)\widetilde{v}_0^{*2}(t)dt+r_0(\tau)\widetilde{Y}^{2}(\tau)\Big|\mathfrak{F}_{0}\right]\nonumber\\
&\quad\,\mbox{}-\mathbb{E}\left[\int_0^{\tau}p_0(t)Y^*(t)\widetilde{Y}(t)+\overline{p}_0(t)\overline{Y}^*(t)\widetilde{\overline{Y}}(t)+v^*_0(t)\widetilde{v}_0(t)dt+r_0(\tau)Y^*(\tau)\widetilde{Y}(\tau)\Big|\mathfrak{F}_{0}\right]\nonumber\\
&\leq -\mathbb{E}\left[\int_0^{\tau}p_0(t)Y^*(t)\widetilde{Y}(t)+\overline{p}_0(t)\overline{Y}^*(t)\widetilde{\overline{Y}}(t)+q_0(t)v^*_0(t)\widetilde{v}_0(t)dt-\widetilde{Y}(\tau)P_{v^*_{0}}(\tau)\Big|\mathfrak{F}_{0}\right].
\end{align}
Noticing that $\widetilde{Y}(0)=\langle \widetilde{\alpha}_1(0),Z_{v^*_{0}}(0)\rangle=\langle \widetilde{\alpha}_1(T),Z_{v^*_{0}}(T)\rangle$=0, we have
\begin{align}\label{52}
&\quad\;\mathbb{E}\left[\widetilde{Y}(\tau)P_{v^*_{0}}(\tau)\Big|\mathfrak{F}_{0}\right]\nonumber\\
&=\mathbb{E}\left[\left(\widetilde{Y}(\tau)P_{v^*_{0}}(\tau)-\widetilde{Y}(0)P_{v^*_{0}}(0)\right)-\left(\langle \widetilde{\alpha}_1(T),Z_{v^*_{0}}(T)\rangle-\langle \widetilde{\alpha}_1(0),Z_{v^*_{0}}(0)\rangle\right)\Big|\mathfrak{F}_{0}\right]\nonumber\\
&=\mathbb{E}\Big[\int_0^{\tau}\widetilde{Y}(t)dP_{v^*_{0}}(t)+P_{v^*_{0}}(t)d\widetilde{Y}(t)+\langle d\widetilde{Y}(t),dP_{v^*_{0}}(t)\rangle-\langle \widetilde{\alpha}_1(t),d Z_{v^*_{0}}(t)\rangle\nonumber\\
&\quad\;\mbox{}
-\langle Z_{v^*_{0}}(t),d\widetilde{\alpha}_1(t)\rangle-\langle d\widetilde{\alpha}_1(t),dZ_{v^*_{0}}(t)\rangle\Big|\mathfrak{F}_{0}\Big]\nonumber\\
&=\mathbb{E}\Big[\int_0^{\tau}\left[p_{0}(t)\widetilde{Y}(t)+\overline{p}_{0}(t)\widetilde{\overline{Y}}(t)\right]Y^*(t)-\left[(l_1,l_1)\widetilde{\alpha}_1+b_0\widetilde{v}_{0}\right]P_{v^*_{0}}(t)+\langle\delta_1(t),Z_{v^*_{0}}(t)\rangle \widetilde{v}_0(t)\nonumber\\
&\quad\;\mbox{}+\left\langle\overline{\delta}_1(t),\overline{Z}_{v^*_{0}}(t)\right\rangle \widetilde{v}_0(t)+
\langle \widetilde{\alpha}_1(t),
\begin{bmatrix}
l_1\\
l_1
\end{bmatrix}
P_{v^*_{0}}(t)\rangle dt\Big|\mathfrak{F}_{0}\Big].
\end{align}
Substituting \eqref{50} and \eqref{52} into \eqref{51} obtains
\begin{align*}
&\quad\;J_0(v^*_1,v^*_2,v^*_0)-J_0(v^*_1,v^*_2,v_0)\\
&\leq \mathbb{E}\Big[\int_0^{\tau}q_0(t)v^*_0(t)\widetilde{v}_0(t)+\left[(l_1,l_1)(t)\widetilde{\alpha}_1(t)-b_0(t)\widetilde{v}_{0}\right]P_{v^*_{0}}(t)+\langle\delta_1(t),Z_{v^*_{0}}(t)\rangle \widetilde{v}_0(t)\\
&\quad\;\mbox{}-\langle\overline{\delta}_1(t),\overline{Z}_{v^*_{0}}(t)\rangle \widetilde{v}_0(t)-\langle \widetilde{\alpha}_1(t),
\begin{bmatrix}
l_1\\
l_1
\end{bmatrix}
(t)P_{v^*_{0}}(t)\rangle dt\Big|\mathfrak{F}_{0}\Big]\\
&=\mathbb{E}\Big[\int_0^{\tau}\left(q_0(t)v^*_0(t)-b_0(t)P_{v^*_{0}}(t)-\langle\delta_1(t),Z_{v^*_{0}}(t)\rangle-\langle\overline{\delta}_1(t),\overline{Z}_{v^*_{0}}(t)\rangle\right)\widetilde{v}_0(t)dt\Big|\mathfrak{F}_{0}\Big]\\
&=0,
\end{align*}
which is the desired result.

According to the above arguments, we have the following theorem.

\begin{thm}\label{53}
The optimal strategy of the leader for Subproblem \ref{88} is
\begin{equation}\label{6}
v_0^*=q_0^{-1}(b_0P_{v^*_{0}}+\langle\delta_1,Z_{v^*_{0}}\rangle+\langle\overline{\delta}_1,\overline{Z}_{v^*_{0}}\rangle).
\end{equation}
Here $(Y^*,\alpha^*_1,\beta^*,P_{v^*_{0}},Q_{v^*_{0}},R_{v^*_{0}},Z_{v^*_{0}})$ is a solution to the following linear MF-FBSDE:
\begin{equation}\label{83}
\begin{cases}
dY^*=\left[(a_1+l_1\psi_1+\gamma^{\mathfrak{F}}d_0)Y^*+(\overline{a}_1+l_1\psi_2+\gamma^{\mathfrak{F}}\overline{d}_0)\overline{Y}^*+(l_1,l_1)\alpha^*_1+b_0v^*_{0}\right]dt\\
\qquad\;\,\:\mbox{}+\left[c_1Y^*+\overline{c}_1\overline{Y}^*\right]dB_t+\left[d_0Y^*+\overline{d}_0\overline{Y}^*\right]dA_t, \quad t\in(0,\tau],\\
d\alpha^*_1=-\left(\mu_1\alpha^*_1+\nu_1\overline{\alpha}^*_1+c_1\beta^*_1+\overline{c}_1\overline{\beta}^*_1+\delta_1v^*_0+\overline{\delta}_1\overline{v}^*_0\right)dt
+\beta^*_1 dB_t, \quad t\in[0,\tau),\\
dP_{v^*_{0}}=\Big\{p_{0}Y^*+\overline{p}_{0}\overline{Y}^*-(a_1+l_1\psi_1+\gamma^{\mathfrak{F}}d_0)P_{v^*_{0}}-(\overline{a}_1+l_1\psi_2)\overline{P}_{v^*_{0}}-\overline{d}_0\mathbb{E}\left[\gamma^{\mathfrak{F}}P_{v^*_{0}}\Big|\mathfrak{F}_{0}\right]\\
\qquad\;\,\:\mbox{}-(c_1Q_{v^*_{0}}+\overline{c}_1\overline{Q}_{v^*_{0}})-(R_{v^*_{0}}\gamma^{\mathfrak{F}}d_0+\overline{d}_0\mathbb{E}\left[\gamma^{\mathfrak{F}}R_{v^*_{0}}\Big|\mathfrak{F}_{0}\right])\Big\}dt+Q_{v^*_{0}}dB_t+R_{v^*_{0}}dA_t, \quad t\in(0,\tau],\\
dZ_{v^*_{0}}=\left(
(l_1,l_1)^T
P_{v^*_{0}}+\mu_1^TZ_{v^*_{0}}+\nu_1^T\overline{Z}_{v^*_{0}}\right)dt+\left(c_1Z_{v^*_{0}}+\overline{c}_1\overline{Z}_{v^*_{0}}\right) dB_t,
\quad t\in[0,\tau),\\
Y^*(0)=x_0,\quad \alpha^*_1(\tau)=\iota(\tau),\quad P_{v^*_{0}}(\tau)=-r_0(\tau)Y^*(\tau),\quad Z_{v^*_{0}}(0)=(0,0)^T.
\end{cases}
\end{equation}
\end{thm}

\begin{remark} If $x_0\in \mathbb{R}$, then the conditional mean-field term $\overline{X}(t)$ reduces to the classical mean-field term.  If $\tau>T$ a.s., then the default event will not happen before time $T$.  Moreover, if $x_0\in \mathbb{R}$ and $\tau>T$ a.s., then Theorem \ref{53} reduces to Theorem 3.4 in \cite{Wang2020A}.
\end{remark}

By choosing $v_0=v_0^*$ in \eqref{82}, $(P_{v^*_{1,1}},P_{v^*_{2,1}})$ is uniquely determined and so is $(v^*_{1,1},v^*_{2,1})$. Combining Theorems \ref{7}, \ref{13} and \ref{53}, we have the following results.
\begin{thm}
The Stackelberg solution of Problem \ref{32} is given by \eqref{80}, \eqref{81} and \eqref{50}.
\end{thm}

In practical situations, the feedback representation of the Stackelberg solution is very useful. However, it is hard to find the feedback representation of the Stackelberg solution of Problem \ref{32} since random coefficients and conditional expectation are involved in \eqref{83}. Nevertheless, the following example shows that we can obtain the feedback representation of the Stackelberg solution in some special cases.
\begin{example}
When $\overline{a}_1=\overline{c}_1=\overline{d}_0\equiv 0$, \eqref{6} and \eqref{83} reduce respectively to
$$
v_0^*=q_0^{-1}(b_0P_{v^*_{0}}+\langle\delta_1,Z_{v^*_{0}}\rangle)
$$
and
\begin{equation*}
\begin{cases}
d\mathcal{Y}=\left[\mathcal{L}_1\mathcal{Y}+\mathcal{L}_2\mathcal{P}\right]dt+c_1\mathcal{Y}dB_t+\mathcal{L}_3\mathcal{Y}dA_t, \quad t\in(0,\tau],\\
d\mathcal{P}=-\left(\mathcal{M}_1\mathcal{P}+\mathcal{M}_2\mathcal{Y}+\mathcal{M}_3\mathcal{R}+c_1\mathcal{Q}
\right)dt
+\mathcal{Q}dB_t+\mathcal{R}dA_t, \quad t\in[0,\tau),\\
\mathcal{Y}(0)=(x_0,0,0)^T,\quad \mathcal{P}(\tau)=\mathcal{G}(\tau)\mathcal{Y}(\tau).
\end{cases}
\end{equation*}
Here $k^{\mathfrak{F}}=a_1+l_1\psi_1+\gamma^{\mathfrak{F}}d_0$ and
$$
\mathcal{Y}=
\begin{bmatrix}
Y^*\\
Z_{v^*_{0}}
\end{bmatrix},\quad
\mathcal{P}=
\begin{bmatrix}
P_{v^*_{0}}\\
\alpha^*_1
\end{bmatrix},\quad
\mathcal{Q}=
\begin{bmatrix}
\beta_1^*\\
Q_{v^*_{0}}
\end{bmatrix},\quad
\mathcal{R}=
\begin{bmatrix}
0\\
R_{v^*_{0}}
\end{bmatrix},
$$
$$
\mathcal{L}_1=
\begin{bmatrix}
k^{\mathfrak{F}} & q_0^{-1}b_0\delta_1^T\\
0 & \mu_1^T
\end{bmatrix},\quad
\mathcal{L}_2=
\begin{bmatrix}
q_0^{-1}b^2_0 &(l_1,l_1)\\
(l_1,l_1)^T &0
\end{bmatrix},\quad
\mathcal{M}_1=
\begin{bmatrix}
k^{\mathfrak{F}}&0 \\
\delta_1q_0^{-1}b_0 &\mu_1
\end{bmatrix},\quad
\mathcal{M}_2
=\begin{bmatrix}
0& q_0^{-1}\delta_1\delta^T_1\\
-p_{0} & 0
\end{bmatrix},
$$
$$
\mathcal{L}_3=
\begin{bmatrix}
d_0 & 0 &0\\
0 & 0 &0\\
0 & 0 &0
\end{bmatrix},\quad
\mathcal{M}_3
=\begin{bmatrix}
0 & 0 &0\\
0 & 0 &0\\
0&0 & \gamma^{\mathfrak{F}}d_0
\end{bmatrix},\quad
\mathcal{G}(\tau)
=\begin{bmatrix}
1& 0&0\\
0& 1 &0\\
0&0& -r_0(\tau)
\end{bmatrix}.
$$
Obviously, $\mathcal{L}_1$ and $\mathcal{M}_1$ are random matrixes with $\mathcal{L}_1=\mathcal{M}_1^T$. Analogously by the method of undetermined coefficients as in Subsection 3.1, suppose that
$$
\mathcal{P}(t)=\widetilde{G}(t)\mathcal{Y}(t),
$$
where $d\mathcal{G}(t)=\varpi_1dt+\varpi_2dB_t+\varpi_3dA_t$. Then, we are able to conclude that $\widetilde{G}(t)$ is a solution to the following stochastic Riccati equation:
\begin{equation}\label{2}
\begin{cases}
d\widetilde{G}=-\left[\widetilde{G}\mathcal{L}_1+\widetilde{G}\mathcal{L}_2\widetilde{G}+\mathcal{L}_1^T\widetilde{G}+\mathcal{M}_2+\mathcal{M}_3\widetilde{G}\mathcal{L}_3
+c_1^2\widetilde{G}+c_1\varpi_2+\mathcal{M}_3\varpi_3\right]dt\\
\qquad\;\;\,+\varpi_2dB_t+\varpi_3\mathcal{Y}dA_t, \quad t\in[0,\tau),\\
\widetilde{G}(\tau)=\mathcal{G}(\tau).
\end{cases}
\end{equation}
Note that \eqref{2} is indeed a matrix-valued nonlinear BSDE with single jump. Thus, by using Theorem 4.1 in \cite{Zhang2018Backward}, the solution to stochastic Riccati equation \eqref{2} has an feedback representation and so does $v_0^*$. By setting $v_0=v_0^*$ in \eqref{4}, $(v^*_{1,1},v^*_{2,1})$ is uniquely determined via \eqref{81} and \eqref{27}. Thus, the feedback representation of Stackelberg solution is obtained.
\end{example}

\section{Conclusions}
This paper focuses on the study of LQ-MF-SSDG with random exit time in the framework of progressive enlargement of filtration, in which the leader is allowed to stop her strategy at a random time. By employing the backward induction method, we obtain the  Stackelberg solution of LQ-MF-SSDG with random exit time. We would like to point out that in the classical mean-field setting, it is hard to derive BSDE for LQ-MF-SSDG because $\mathbb{E}[a(t)Y(t)\mathbb{E}\left[X(t)\right]]\neq\mathbb{E}[X(t)\left[a(t)Y(t)\right]]$ when $a(t)$ is an $\mathfrak{F}_t$-adapted process. However, by employing the fact that $\mathbb{E}[a(t)Y(t)\mathbb{E}\left[X(t)\Big|\mathfrak{F}_s\right]\Big|\mathfrak{F}_s]=\mathbb{E}[X(t)\left[a(t)Y(t)\Big|\mathfrak{F}_s\right]\Big|\mathfrak{F}_s]\;(t\geq s)$,  we can obtain BSDE for LQ-MF-SSDG in the conditional mean-field setting and so the Pontryagin maximum principle remains valid.

Recall that the problem of the feedback representation of $v_0^*$ in \eqref{83} is unsolved. Therefore, how to find the feedback representation of Stackelberg solution for LQ-MF-SSDG with random exit time is a problem worth investigating. We also mention that the backward induction method becomes invalid for solving the Stackelberg solution when the followers are allowed to stop their strategies at a random time.  This problem is closely related to the multi-stage Stackelberg game, which is usually known as a very difficult problem to be solved. Thus, it would be interesting to derive the Stackelberg solution of LQ-MF-SSDG when one of the followers is allowed to exit the game at a random time. We will consider these problems in future research.


\end{document}